\documentclass[a4paper,12pt,twoside]{article}

\usepackage{doi}

\usepackage[dvipsnames]{xcolor} 

\usepackage{hyperref} 
\hypersetup{breaklinks, colorlinks,
    linkcolor = {Blue},
    citecolor = {Blue},
    urlcolor  = {Blue}
}

\usepackage{amsmath,amssymb}
\usepackage{mathtools} 
\usepackage{alphabeta} 

\usepackage{amsthm, 
}

\usepackage[numbers,sort&compress]{natbib}
\usepackage[text={6.5in,9.5in},centering]{geometry}

\usepackage[inline,shortlabels]{enumitem} 
\renewlist{enumerate}{enumerate}{2}
\setlist[enumerate,1]{label=\textup{(\alph*)},
ref={\alph*}, align=left, labelsep=0.5ex, leftmargin=*}
\setlist[enumerate,2]{label=\textup{({\roman*})},
ref={\roman*}, align=right, labelsep*=1ex, widest={(ii)},  
leftmargin=5.4ex}

\DeclareSymbolFont{bbold}{U}{bbold}{m}{n}
\DeclareMathSymbol{\bbone}{\mathord}{bbold}{"31} 

\theoremstyle{plain}
\newtheorem{theorem}{Theorem}[section]
\newtheorem{corollary}{Corollary}[section]
\newtheorem{lemma}{Lemma}[section]
\newtheorem{proposition}{Proposition}[section]

\newtheoremstyle{boldremex}
    {\dimexpr\topsep/2\relax} 
    {\dimexpr\topsep/2\relax} 
    {}          
    {}          
    {\bfseries} 
    {.}         
    {.5em}      
    {}          
    
\theoremstyle{boldremex}
\newtheorem{remark}{Remark}[section]

\makeatletter
\renewenvironment{proof}[1][\proofname]{%
   \par\pushQED{\qed}\normalfont%
   \topsep6\p@\@plus6\p@\relax
   \trivlist\item[\hskip\labelsep\bfseries#1\@addpunct{.}]%
   \ignorespaces
}{%
   \popQED\endtrivlist\@endpefalse
}
\makeatother

\newcommand{\CC}{\mathbb{C}}
\newcommand{\RR}{\mathbb{R}}
\newcommand{\NN}{\mathbb{N}}
\newcommand{\ZZ}{\mathbb{Z}}
\newcommand{\Zpl}{\ZZ_+}

\newcommand{\set}[1]{\underline{#1}}

\newcommand{\dd}{{\mathrm{d}}}
\newcommand{\ee}{{\mathrm{e}}}
\newcommand{\ii}{{\mathrm{i}}}
\newcommand{\pii}{{\text{\pi}}}
 
\newcommand{\abs}[1]{|#1|}                   

\newcommand{\ABS}[1]{\Bigl|#1\Bigr|}           

\newcommand{\card}[1]{|#1|}                 

\newcommand{\bfcdot}{{\boldsymbol{\cdot}}}

\newcommand{\EE}{\mathrm{E}} 

\newcommand{\floor}[1]{{\lfloor #1 \rfloor}} 

\newcommand{\Var}{\mathrm{Var}}

\newcommand{\newatop}[2]{\genfrac{}{}{0pt}{}{\scriptstyle #1}{\scriptstyle #2}}

\scrollmode

\begin{document}
\title{On the accuracy in a combinatorial 
central limit theorem: the characteristic function method}
\author{
Bero Roos\footnote{
Postal address:
FB IV -- Dept.\ of Mathematics, 
University of Trier, 
54286 Trier, Germany. 
E-mail: \texttt{bero.roos@uni-trier.de}}\medskip\\
University of Trier 
\date{
}
}
\maketitle
\begin{abstract}
The aim of this paper is to present a new proof of an explicit version 
of the Berry-Ess\'{e}en type inequality of Bolthausen 
(Zeitschrift f\"ur Wahrscheinlichkeitstheorie und Verwandte 
Gebiete, 66, 379--386, 1984). The literature already provides 
several proofs of it using variants of 
Stein's method. The characteristic function method has also been 
applied but led only to weaker results. 
In this paper, we show how to overcome the difficulties of this 
method by using a new identity for permanents of complex matrices 
in combination with a recently proved inequality for the 
characteristic function of the approximated distribution.
 
\end{abstract}
\section{Introduction and main result} \label{s86256}
The characteristic function method has shown to be very useful in 
the approximation of probability distributions, when enough information 
on the characteristic functions of the considered distributions is
available; for instance, the best constant to date in 
the Berry-Ess\'{e}en theorem for sums of independent real-valued random 
variables was obtained by \citet{Shevtsova2013}. 
In the present paper, we use the characteristic function method to 
give a new proof of an explicit version of the Berry-Ess\'{e}en type 
inequality of \citet{MR751577} in a combinatorial central limit 
theorem. Our main focus lies on the basic method while trying to obtain 
reasonable constants and to avoid complex reasoning. 
So we believe that the constants in our main
result (see Theorem \ref{t798326}) could be further improved in future 
by using refinements of the method. 
 
In what follows, we need some notation. 
Let $n\in\NN=\{1,2,3,\dots\}$ be a natural number with $n\geq2$, 
$A=(a_{j,r})\in\RR^{n\times n}$ be a real-valued $n\times n$ matrix,
and $\pi=(\pi(1),\dots,\pi(n))$ be a uniformly distributed random 
permutation of the set $\set{n}:=\{1,\dots,n\}$, that is 
$P(\pi=j)=\frac{1}{n!}$ for all 
$j\in\set{n}_{\neq}^n:=\{(j_1,\dots,j_n)\,|\,j_1,\dots,j_n\in\set{n}
\mbox{ pairwise distinct}\}$. 
Sometimes, we also write $(j(1),\dots,j(n))=(j_1,\dots,j_n)$.
Further, let  
\begin{align*}
S_n=\sum_{j=1}^na_{j,\pi(j)}.
\end{align*}
It is well-known that 
\begin{align}\label{e8625987}
\mu:=\EE S_n=na_{\bfcdot,\bfcdot}\quad \mbox{and}\quad
\sigma
:=(\Var S_n)^{1/2}
=\Bigl(\frac{1}{n-1}
\sum_{(j,r)\in\set{n}^2}\widetilde{a}_{j,r}^2\Bigr)^{1/2},
\end{align}
where $\set{n}^2=\{(j,r)\,|\,j,r\in\set{n}\}$ and
\begin{align*}
\widetilde{a}_{j,r}=a_{j,r}-a_{\bfcdot,r}-a_{j,\bfcdot}
+a_{\bfcdot,\bfcdot},\quad
a_{\bfcdot,r}=\frac{1}{n}\sum_{k=1}^na_{k,r}, \quad
a_{j,\bfcdot}=\frac{1}{n}\sum_{s=1}^na_{j,s}, \quad
a_{\bfcdot,\bfcdot}=\frac{1}{n^2}\sum_{(k,s)\in\set{n}^2}a_{k,s}
\end{align*}
for all $(j,r)\in\set{n}^2$, see \citet[Theorem 2]{MR44058}. 
We always assume that $\sigma^2>0$. 
In \citet[formula (89)]{MR2349578} 
(see also \citet[formula (4.106)]{MR2732624}), another formula for 
$\sigma^2$ can be found, which reads as follows: 
\begin{align}\label{e872897}
\sigma^2
=\frac{1}{4n^2(n-1)}
\sum_{(j,k)\in\set{n}_{\neq}^2}\sum_{(r,s)\in\set{n}_{\neq}^2}
b_{j,k,r,s}^2,
\end{align}
where 
$\set{n}_{\neq}^2=\{(j,r)\in\set{n}^2\,|\,j,r\mbox{ distinct}\}$
and
\begin{align}
b_{j,k,r,s}
&=a_{j,r}-a_{k,r}-a_{j,s}+a_{k,s}
=\widetilde{a}_{j,r}-\widetilde{a}_{k,r}-\widetilde{a}_{j,s}
+\widetilde{a}_{k,s}\label{e97257}
\end{align}
for all $j,k,r,s\in\set{n}$. 
For a slightly more general assertion, see \eqref{e9726578}, which is 
shown in Section~\ref{s7145876}. 
We note that $b_{j,k,r,s}=-b_{k,j,r,s}=-b_{j,k,s,r}$ 
for all $j,k,r,s\in\set{n}$ and that $b_{j,k,r,s}=0$ if $j=k$ or 
$r=s$. Furthermore,
\begin{align}\label{e687165}
\widetilde{a}_{j,r}
=\frac{1}{n^2}\sum_{(k,s)\in\set{n}^2}b_{j,k,r,s}
\quad\mbox{ for }(j,r)\in\set{n}^2. 
\end{align}

The statistic $S_n$ is of interest in the theory of rank tests,
see \citet{MR1680991}.
Central limit theorems for 
\begin{align}\label{e826653}
S_n^*:=\frac{S_n-\mu}{\sigma}
=\frac{1}{\sigma}\sum_{j=1}^n\widetilde{a}_{j,\pi(j)}
\end{align}
under various conditions have been proved by
\citet{MR11424},
\citet{MR44058},
\citet{MR89560}, and
\citet{MR130707}; 
see also the references therein.
The treatment of corresponding estimates of accuracy has been
more difficult. Two major approaches have been used here to bound 
the supremum distance between the corresponding distribution 
functions.
Methods using characteristic functions led to bounds of the right 
order only under additional conditions, which, however, we shall not 
discuss here in detail; but let us mention that,  while 
\citet{MR418187} combined a classical approach with combinatorial 
arguments such as the M\"obius inversion formula, 
\citet{MR0458711,MR550444} and \citet{MR672298} 
considered the case of $a_{j,r}=a'_{j}a''_r$, $(a'_j,a''_r\in\RR)$ 
and used an intermediate approximation by a sum of independent 
random variables. The principal difficulty here is to find a
sufficiently weak condition for the validity of a sufficiently good 
upper bound for the modulus of the difference of the characteristic 
function of $S_n^*$ and that of the standard normal law. 

On the other hand, variants of Stein's method have been used, e.g., 
see \citet{MR478291}, and the works cited below. 
The first satisfying result was established by
\citet{MR751577} using Stein's method together with an induction.
It says that 
\begin{align}\label{e2765776}
\Delta
:=\sup_{x\in\RR}\abs{P(S_n^*\leq x)-\Phi(x)}
\leq \frac{C_0}{n\sigma^3}\sum_{(j,r)\in\set{n}^2}
\abs{\widetilde{a}_{j,r}}^3.
\end{align}
where $\Phi$ is the distribution function of the standard normal 
distribution and $C_0$ is an absolute constant.  
The inequality in \eqref{e2765776} is of Lyapunov type, since 
the right-hand side can be written as 
$\frac{C_0}{\sigma^{3}}\sum_{j=1}^n\EE|\widetilde{a}_{j,\pi(j)}|^3$.
It has the order $n^{-1/2}$, 
if $\frac{1}{\sqrt{n}\sigma^3}\sum_{(j,r)\in\set{n}^2}
\abs{\widetilde{a}_{j,r}}^3$ can be estimated from above by an 
absolute constant. For instance, \eqref{e8625987} implies that this 
is true if two absolute 
constants $C_0',C_0''\in(0,\infty)$ exist such that 
$C_0'\leq \abs{\widetilde{a}_{j,r}}\leq C_0''$ for all
$(j,r)\in\set{n}^2$. 

Unfortunately, the constant $C_0$
in \eqref{e2765776} was not explicitly given in \cite{MR751577}.
Much later, employing the zero bias version of Stein's method
used in \citet{MR2157512}, an upper bound
of worse order than the one in \eqref{e2765776} but with an explicit 
constant was shown in \citet[Theorem 6.1, p.~168]{MR2732624}), that 
is, if $n\geq3$, then
\begin{align*} 
\Delta\leq\frac{ 16.3}{\sigma}\max_{(j,r)\in\set{n}^2}\abs{
\widetilde{a}_{j,r}}.
\end{align*}
In \citet{MR3322321}, Stein's method of exchangeable 
pairs and a concentration inequality was used to prove that 
$C_0\leq 451$. Using a zero bias version of Stein's method together 
with an induction, \citet{Thanh2013} showed that $C_0\leq90$. 

There is an improvement of \eqref{e2765776}.
Combining this inequality with a classical truncation technique,  
\citet{MR3199896} showed his Theorem 4, which, in the present 
situation, says that 
\begin{align}\label{e8259698}
\Delta
\leq \widetilde{C}_0\frac{n-1}{n\sigma^2}
\widetilde{\gamma}\Bigl(\frac{1}{\sigma}\Bigr),
\end{align}
where $\widetilde{C}_0=\max\{1709,50C_0+6\}$, 
$C_0$ is the constant in \eqref{e2765776}, and
\begin{align}\label{e872766}
\widetilde{\gamma}(x)
&=\frac{1}{n-1}\sum_{(j,r)\in\set{n}^2}
\widetilde{a}_{j,r}^2\min\{1,\abs{x\widetilde{a}_{j,r}}\}\quad
\mbox{ for }x\in\RR.
\end{align}
In view of the minimum term in \eqref{e872766}, we see that 
\eqref{e8259698} immediately implies Motoo's \cite{MR89560} 
combinatorial central limit theorem. 

It should be mentioned that several authors considered the more 
general situation, where
$A$ is replaced by a matrix of independent random variables 
being also independent of $\pi$
(e.g., see \cite{MR418187}, \cite{MR478291}, \cite{MR3322321}, and
\cite{MR3199896}). 
Further, estimates of higher accuracy have been shown
using an Edgeworth expansion, e.g.\ see \citet{MR696072} and
\citet{MR1015140}.
It seems that the method presented here can also 
be applied to these kinds of generalization.

Our main result is the next theorem, 
which is comparable to \eqref{e8259698} and is proved with 
the characteristic function method. We need further notation. Let 
\begin{align}\label{e78256}
\gamma(x)
&=\frac{1}{n^2(n-1)}
\sum_{(j,k)\in\set{n}_{\neq}^2}\sum_{(r,s)\in\set{n}_{\neq}^2}
b_{j,k,r,s}^2\min\{1,\abs{xb_{j,k,r,s}}\}
\quad\mbox{for }x\in\RR. 
\end{align}
\begin{theorem} \label{t798326}
Let $C_1=15.84$, $C_2=0.65$, that is, $C_1C_2\leq 10.3$. Then
\begin{align}\label{e71545}
\Delta
\leq \frac{C_1}{\sigma^2}\gamma\Bigl(\frac{C_2}{\sigma}\Bigr).
\end{align}
\end{theorem}
For the comparison of \eqref{e71545} with  \eqref{e2765776}
and \eqref{e8259698}, we shall 
use the following lemma, which provides upper and lower bounds  of
$\gamma(x)$
and is proved in Section \ref{s7145876}. Here, we define
\begin{align*}
\delta
=\frac{1}{n^2(n-1)}
\sum_{(j,k)\in\set{n}_{\neq}^2}\sum_{(r,s)\in\set{n}_{\neq}^2}
\abs{b_{j,k,r,s}}^3.
\end{align*}

\begin{lemma}\label{l1876565}
Let $x\in\RR\setminus\{0\}$. For arbitrary $y\in(0,1)$, we then have 
\begin{align}
\Bigl(1-y^2\Bigl(\frac{n-1}{n}\Bigr)^2\Bigr)
\widetilde{\gamma}(xy)
&\leq\gamma(x)\leq 16 \widetilde{\gamma}(x),\label{e2976787}\\
4\Bigl(\sigma^2-\frac{n-1}{27x^2}\Bigr)
&\leq\gamma(x)
\leq \min\{4\sigma^2, \abs{x}\delta\}.\label{e78257}
\end{align}
\end{lemma}

\begin{remark} 
\begin{enumerate}

\item 
The inequalities in \eqref{e2976787} and the fact that 
$y\widetilde{\gamma}(x)\leq \widetilde{\gamma}(xy)$ for 
$y\in(0,1)$ imply that $\gamma(x)$ and $\widetilde{\gamma}(x)$ are 
of the same order. 

\item The constant $16$ on the right-hand side of \eqref{e2976787} is
optimal, as can be shown by considering the example $n=2$ and 
$A=\begin{psmallmatrix*}[r]
t&-t\\
-t&t
\end{psmallmatrix*}$ 
for $t\in\RR$, giving 
$\gamma(x)=16t^2\min\{1,\abs{4xt}\}$ and
$\widetilde{\gamma}(x)=4t^2\min\{1,\abs{xt}\}$.

\item 
From \eqref{e71545} and the second inequality in 
\eqref{e2976787}, we can derive an 
inequality of the form \eqref{e2765776} or \eqref{e8259698}. 
In fact, 
\begin{align}
\Delta\leq 
\frac{C_1}{\sigma^2}\gamma\Bigl(\frac{C_2}{\sigma}\Bigr)
\leq 16\frac{C_1}{\sigma^2}
\widetilde{\gamma}\Bigl(\frac{C_2}{\sigma}\Bigr)
\leq \frac{164.6}{(n-1)\sigma^3}
\sum_{(j,r)\in\set{n}^2}\abs{\widetilde{a}_{j,r}}^3.\label{e62796}
\end{align}
In particular, the constants in the second bound are better than 
those of \eqref{e8259698}. However, the constant $164.6$ 
in the third bound in \eqref{e62796} is somewhat large and cannot 
compete with Thanh's \cite{Thanh2013} constant $90$. 
On the other hand, the bounds in \eqref{e71545} and \cite{Thanh2013}
are not easily comparable because of the minimum term in 
\eqref{e78256}. 
We believe that the method of this paper could be refined to 
produce smaller constants. For instance, this could be achieved by 
replacing \eqref{e7165438} below by a more accurate smoothing 
lemma. The first two inequalities in \eqref{e62796} show that the 
characteristic function method can be used to give a proof of 
Motoo's \cite{MR89560} combinatorial central limit theorem. 
\end{enumerate}
\end{remark}

Let us now consider an example of sampling without replacement, where 
$S_n$ is the sum of $m\in\set{n}$ of real numbers $c_1,\dots,c_n$ 
drawn uniformly at random without replacement. 
\begin{corollary}\label{c2765} 
Let $c_1,\dots,c_n\in\RR$ and $m\in\set{n}$. 
Further, let $a_{j,r}=c_r$, if $(j,r)\in\set{m}\times \set{n}$, and 
$a_{j,r}=0$ otherwise. Then we have
$\mu
=mc_{\bfcdot}$ and
$\sigma^2
=\frac{m(n-m)}{n(n-1)}\sum_{r\in\set{n}}(c_r-c_{\bfcdot})^2$,
where $c_{\bfcdot}=\frac{1}{n}\sum_{r=1}^nc_r$.
Assuming that $\sigma^2>0$ as usual, \eqref{e71545} holds with
\begin{align}\label{e876245}
\Delta
\leq 
\frac{C_1}{\sigma^2}\gamma\Bigl(\frac{C_2}{\sigma}\Bigr)
&=\frac{2C_1m(n-m)}{n^2(n-1)\sigma^2}
\sum_{(r,s)\in\set{n}_{\neq}^2}(c_r-c_s)^2
\min\Bigl\{1,\frac{C_2}{\sigma}\abs{c_r-c_s}\Bigr\}.
\end{align}
\end{corollary}
\begin{remark} 
Let the assumptions of Corollary \ref{c2765} hold. From 
\eqref{e876245}, one can derive a bound for $\Delta$, which has 
the same order as the bound in \citet[formula (1)]{MR545097}.
It should be mentioned that, in \cite{MR545097}, 
the variances of the two considered distributions differ slightly. 
However there are cases, where equal variances lead to a more precise
approximation, see \citet[Theorem 1.3]{MR3717995}. 
\end{remark}

The rest of the paper is structured as follows. 
The next section is devoted to the proof of Theorem \ref{t798326},
which is heavily based on an explicit inequality for the difference 
between the characteristic functions of the considered distributions 
(see Proposition \ref{p719876}) and
a variant of Ess\'{e}en's fundamental inequality 
(see \eqref{e7165438}). 
All needed auxiliary results will be proved in later sections. 
In particular, Section \ref{s5254} contains the proof of 
Proposition \ref{p719876}, which requires some preparations in form 
of a useful identity for permanents (see Proposition \ref{p874976} 
in Section \ref{s8257}), 
an upper bound for the characteristic function of $S_n$
(see Lemma \ref{l765245}), and other results. 
Section \ref{s7145876} is devoted to the remaining proofs. 

\section{Proof of Theorem \ref{t798326}}
Let the assumptions of Section \ref{s86256} hold. We proceed with 
the preparations for the proof of Theorem \ref{t798326}. Let 
$\varphi$ be the characteristic function of 
$S_n$, that is,
\begin{align*}
\varphi(t)=\EE\ee^{\ii t S_n}
\end{align*}
for $t\in\RR$. 
For a set $M$, let $\bbone_{M}(x)=1$, if $x\in M$, and 
$\bbone_{M}(x)=0$ otherwise. 

\begin{lemma} 
We have 
\begin{gather}
\kappa:=\sup_{x\in\RR} \frac{\cos(x)-1+x^2/2}{\abs{x}^3}
=0.09916191\dots ,\label{e982758}
\end{gather}
where the supremum is attained at the unique point $x_0=3.99589\dots$.
Here we set $\frac{0}{0}=0$. 
\end{lemma}
The proof is given in \citet[formula (12)]{MR350809}. 
The next result is shown in Section \ref{s5254}.
\begin{proposition} \label{p719876}
Let $t\in\RR$,
\begin{align}\label{e29876898}
h_{\ell}(t)
=\min\Bigl\{1,\exp\Bigl(\ell-\frac{n-\ell-1}{4(n-1)}t^2
\Bigl(\sigma^2-\frac{1}{4}\gamma(2\kappa t)\Bigr)\Bigr)
\Bigr\}\bbone_{[\ell,\infty)}(n)
\end{align}
for $\ell\in[0,\infty)$, where $\kappa$ is given in \eqref{e982758}. 
Then  
\begin{align}
\begin{split}
\abs{\varphi(t)-\ee^{\ii t\mu-\sigma^2 t^2/2}}
&\leq \int_0^1 t^2u
\Bigl(\frac{1}{2}h_2(tu)\gamma\Bigl(\frac{tu}{4}\Bigr)
+\frac{2(n-2)}{n(n-1)}h_3(tu)\gamma\Bigl(\frac{tu}{2}\Bigr)\\
&\quad{}+
\frac{(n-2)(n-3)}{n(n-1)}h_4(tu)\gamma\Bigl(\frac{tu}{2}\Bigr)\Bigr)
\exp\Bigl(-(1-u^2)\frac{\sigma^2t^2}{2}\Bigr)\,\dd u.
\end{split}\label{e8616565}
\end{align}
\end{proposition}

We note that $h_3(t)=0$ for $n=2$ and $h_4(t)=0$ for $n\in\{2,3\}$. 
Furthermore, \eqref{e78257} implies that 
$\sigma^2- \frac{1}{4}\gamma(2\kappa t)\geq0$ for all $t\in\RR$. 

\begin{remark} 
\begin{enumerate}

\item 
Let the assumptions of Proposition \ref{p719876} hold. Then
\eqref{e8616565} implies an inequality without integral, that is,
\begin{align}
\abs{\varphi(t)-\ee^{\ii t\mu-\sigma^2 t^2/2}}
&\leq \frac{t^2}{4}\gamma\Bigl(\frac{t}{6}\Bigr)h_2(t)
+\frac{(n-2)t^2}{n(n-1)}\gamma\Bigl(\frac{t}{3}\Bigr)h_3(t)
+\frac{(n-2)(n-3)t^2}{2n(n-1)}\gamma\Bigl(\frac{t}{3}\Bigr)h_4(t),
\label{e8126565}
\end{align} 
which can easily be shown by using that 
$\int_0^1u\gamma(xu)\,\dd u\leq \frac{1}{2}\gamma(\frac{2x}{3})$
for $x\in\RR$ and 
\begin{align*} 
h_\ell(tu)\exp\Bigl(-(1-u^2)\frac{\sigma^2t^2}{2}\Bigr)
&\leq h_\ell(t)\quad 
\mbox{ for }\ell\in[0,\infty)\mbox{ and } u\in[0,1].
\end{align*}
If $n\geq 6$, then \eqref{e8126565} implies that 
\begin{align}\label{e7682658}
\abs{\varphi(t)-\ee^{\ii t\mu-\sigma^2 t^2/2}}
\leq 32 t^2\gamma\Bigl( \frac{t}{3}\Bigr)
\exp\Bigl(-\frac{t^2}{20}\Bigl(\sigma^2-\frac{1}{4}\gamma(2\kappa t)
\Bigr)\Bigr).
\end{align}
The inequalities in \eqref{e8126565} or \eqref{e7682658} can be used 
to prove an inequality similar to that in Theorem \ref{t798326} but 
with worse constants. 

\item For $t\in\RR$,  
\begin{align*}
\varphi(t)
&=\frac{1}{n!}\sum_{r\in\set{n}_{\neq}^n}
\prod_{j=1}^n\exp(\ii ta_{j,r(j)}),
\end{align*}
is the  permanent of the matrix 
$(\exp(\ii ta_{j,r}))\in\CC^{n\times n}$ divided by $n!$. 
Therefore, inequalities \eqref{e8616565}--\eqref{e7682658}
are, in fact, explicit approximation inequalities 
for permanents with entries lying on the unit circle in the complex 
domain.
\end{enumerate}
\end{remark}
The proof of the following lemma is given in
Section~\ref{s7145876}.

\begin{lemma}\label{l28766}
Let $c\in(0,1/2)$, $\widetilde{c}=\sqrt{1-2c}$, $x\in\RR$, and 
$\gamma(x)$ be as in \eqref{e78256}. Then 
\begin{align}
\begin{split}
\lefteqn{\int_0^\infty\int_0^1tu\gamma(xtu)
\exp\Bigl(-ct^2u^2-(1-u^2)\frac{t^2}{2}\Bigr)\,\dd u\,\dd t}\\
&\leq \frac{-\log(2c)}{2\widetilde{c}^2} 
\gamma\Bigl(\frac{\sqrt{\pii}x}{-\log(2c)\sqrt{c}}
\Bigl(1-\frac{\sqrt{2c}}{\widetilde{c}}
\arcsin(\widetilde{c})\Bigr)\Bigr). 
\end{split}
\label{e218153}
\end{align}
\end{lemma}
\begin{proof}[Proof of Theorem \ref{t798326}] 
In view of \eqref{e826653}, without loss of generality we may assume 
that $\mu=0$ and $\sigma=1$. 
Let $C_4\in(0,\infty)$ and $m\in\NN$ such that $4\leq m$ and 
$\sqrt{\frac{m-1}{27}}<C_4$. 
In particular $0<27C_4^2-m+1$ and 
$C_3:=\frac{27C_4^2}{4(27C_4^2-m+1)}\in(0,\infty)$. 
Let $C_5,C_6\in(0,\infty)$ such that $C_5C_6\in(\frac{1}{4},\infty)$. 
First, we assume that $2\leq n\leq m$. 
Then the first inequality in \eqref{e78257} implies that 
\begin{align}\label{e2523465}
\Delta
\leq 1
\leq \frac{27C_4^2}{4(27C_4^2-n+1)}\gamma(C_4)
\leq C_3\gamma(C_4). 
\end{align}
Let us now assume that $n\geq m+1$.
Let $T=\frac{1}{C_5\gamma(2\kappa C_6)}$. If $T<C_6$ then 
\begin{align}
\Delta
\leq 1
\leq C_5C_6\gamma(2\kappa C_6). \label{e2523466}
\end{align}
We now assume that  $T\geq C_6$ and that $t\in[0,T]$. Then 
$\gamma(2\kappa tu)
\leq \frac{T}{C_6}\gamma(2\kappa C_6)
=\frac{1}{C_5C_6}$ for all $u\in[0,1]$.
We shall use Proposition \ref{p719876}, which gives
\begin{align*}
\abs{\varphi(t)-\ee^{-t^2/2}}
&\leq \int_0^1 t^2u
\Bigl(\frac{1}{2}h_2(tu)\gamma\Bigl(\frac{tu}{4}\Bigr)
+\frac{2(n-2)}{n(n-1)}h_3(tu)\gamma\Bigl(\frac{tu}{2}\Bigr)\\
&\quad{}+
\frac{(n-2)(n-3)}{n(n-1)}h_4(tu)\gamma\Bigl(\frac{tu}{2}\Bigr)\Bigr)
\exp\Bigl(-(1-u^2)\frac{t^2}{2}\Bigr)\,\dd u,
\end{align*}
where, for $\ell\in\{2,3,4\}$,  
\begin{align*}
h_{\ell}(tu)
&\leq \exp\Bigl(\ell-\frac{n-\ell-1}{4(n-1)}t^2u^2
\Bigl(1-\frac{1}{4}\gamma(2\kappa tu)\Bigr)\Bigr)
\leq \exp(\ell-\vartheta_\ell t^2u^2\} 
\end{align*}
and
$\vartheta_\ell
=\frac{m-\ell}{4m}(1-\frac{1}{4C_5C_6})$.
Let $\widetilde{\vartheta}_\ell=\sqrt{1-2\vartheta_\ell}$ and
\begin{align*}
D_\ell
&=\frac{\sqrt{\pii}}{-\log(2\vartheta_\ell)
\sqrt{\vartheta_\ell}}
\Bigl(1-\frac{\sqrt{2\vartheta_\ell}
}{\widetilde{\vartheta}_\ell}
\arcsin(\widetilde{\vartheta}_\ell)\Bigr).
\end{align*}
Using Lemma \ref{l28766}, we get 
\begin{align*}
\lefteqn{\int_0^T\int_0^1 \frac{tu}{2}h_2(tu)
\gamma\Bigl(\frac{tu}{4}\Bigr)
\exp\Bigl(-(1-u^2)\frac{t^2}{2}\Bigr)\,\dd u\,\dd t}\nonumber\\
&\leq \frac{\ee^2}{2}
\int_0^\infty\int_0^1 tu\gamma\Bigl(\frac{tu}{4}\Bigr)
\exp\Bigl(-\vartheta_2t^2u^2 
-(1-u^2)\frac{t^2}{2}\Bigr)\,\dd u\,\dd t\nonumber\\
&\leq \frac{\ee^2}{4}
\frac{(-\log(2\vartheta_2))}{\widetilde{\vartheta}_2^2} 
\gamma\Bigl(\frac{\sqrt{\pii}}{-4\log(2\vartheta_2)
\sqrt{\vartheta_2}}
\Bigl(1-\frac{\sqrt{2\vartheta_2}
}{\widetilde{\vartheta}_2}
\arcsin(\widetilde{\vartheta}_2)\Bigr)\Bigr)\nonumber\\
&=\frac{\ee^2}{4}
\frac{(-\log(2\vartheta_2))}{\widetilde{\vartheta}_2^2} 
\gamma\Bigl(\frac{D_2}{4}\Bigr), 
\end{align*}
\begin{align*}
\lefteqn{\int_0^T\int_0^1 tu
\frac{2(n-2)}{n(n-1)}h_3(tu)\gamma\Bigl(\frac{tu}{2}\Bigr)
\exp\Bigl(-(1-u^2)\frac{t^2}{2}\Bigr)\,\dd u\,\dd t}\nonumber\\
&\leq\frac{2\ee^3(m-1)}{m(m+1)}\int_0^\infty\int_0^1 tu
\gamma\Bigl(\frac{tu}{2}\Bigr)\exp\Bigl(-\vartheta_3t^2u^2
-(1-u^2)\frac{t^2}{2}\Bigr)\,\dd u\,\dd t\nonumber\\
&\leq \frac{2\ee^3(m-1)}{m(m+1)}
\frac{(-\log(2\vartheta_3))}{2\widetilde{\vartheta}_3^2} 
\gamma\Bigl(\frac{\sqrt{\pii}}{-2\log(2\vartheta_3)\sqrt{\vartheta_3}}
\Bigl(1-\frac{\sqrt{2\vartheta_3}}{\widetilde{\vartheta}_3}
\arcsin(\widetilde{\vartheta}_3)\Bigr)\Bigr)\nonumber\\
&=\frac{\ee^3(m-1)}{m(m+1)}
\frac{(-\log(2\vartheta_3))}{\widetilde{\vartheta}_3^2} 
\gamma\Bigl(\frac{D_3}{2}\Bigr),
\end{align*}
and
\begin{align*}
\lefteqn{\int_0^T\int_0^1 tu
\frac{(n-2)(n-3)}{n(n-1)}h_4(tu)\gamma\Bigl(\frac{tu}{2}\Bigr)
\exp\Bigl(-(1-u^2)\frac{t^2}{2}\Bigr)\,\dd u\,\dd t}\nonumber\\
&\leq \ee^4\int_0^\infty\int_0^1 tu\gamma\Bigl(\frac{tu}{2}\Bigr)
\exp\Bigl(-\vartheta_4t^2u^2
-(1-u^2)\frac{t^2}{2}\Bigr)\,\dd u\,\dd t\nonumber\\
&\leq \ee^4
\frac{(-\log(2\vartheta_4))}{2\widetilde{\vartheta}_4^2} 
\gamma\Bigl(\frac{\sqrt{\pii}}{-2\log(2\vartheta_4)
\sqrt{\vartheta_4}}
\Bigl(1-\frac{\sqrt{2\vartheta_4}}{\widetilde{\vartheta}_4}
\arcsin(\widetilde{\vartheta}_4)\Bigr)\Bigr)\nonumber\\
&=\ee^4
\frac{(-\log(2\vartheta_4))}{2\widetilde{\vartheta}_4^2} 
\gamma\Bigl(\frac{D_4}{2}\Bigr). 
\end{align*}
In view of the general Lemma 12.2 in \citet[p.~101--103]{MR3396213}, 
we see that the following variant of  Ess\'{e}en's 
\cite[Theorem 2a, p.\ 32]{MR0014626} smoothing inequality is valid:
\begin{align}\label{e7165438}
\Delta
&\leq \frac{1}{\pii w}\int_{0}^T
\frac{\abs{\varphi(t)-\ee^{-t^2/2}}}{t}\,\dd t
+\frac{(1+w)v(w)}{\sqrt{2\pii}wT},
\end{align}
where $w\in(0,1)$ is arbitrary and $v(w)\in(0,\infty)$ is defined 
by the equation 
\begin{align*}
\frac{1+w}{2}
=\frac{2}{\pii}\int_0^{v(w)}\frac{\sin^2(x)}{x^2}\,\dd x. 
\end{align*}
Combining the inequalities above together with the simple fact that
\begin{align*}
x_1\gamma(y_1)+x_2\gamma(y_2)
\leq (x_1+x_2)\gamma\Bigl(\frac{x_1y_1+x_2y_2}{x_1+x_2}\Bigr),
\end{align*}
for $x_1,x_2,y_1,y_2\in[0,\infty)$, we obtain 
\begin{align}
\Delta
&\leq \frac{\ee^2}{\pii w}
\frac{(-\log(2\vartheta_2))}{4\widetilde{\vartheta}_2^2} 
\gamma\Bigl(\frac{D_2}{4}\Bigr)
+\frac{\ee^3(m-1)}{\pii w m(m+1)}
\frac{(-\log(2\vartheta_3))}{\widetilde{\vartheta}_3^2} 
\gamma\Bigl(\frac{D_3}{2}\Bigr)
+\ee^4\frac{(-\log(2\vartheta_4))}{2\pii w\widetilde{\vartheta}_4^2} 
\gamma\Bigl(\frac{D_4}{2}\Bigr)\nonumber\\
&\quad{}+\frac{(1+w)v(w)}{\sqrt{2\pii}w}C_5\gamma(2\kappa C_6)
\nonumber\\
&\leq C_7\gamma\Bigl(\frac{C_8}{C_7}\Bigr), \label{e2523467}
\end{align}
where 
\begin{align*}
C_7
&=\frac{\ee^2}{\pii w}
\frac{(-\log(2\vartheta_2))}{4\widetilde{\vartheta}_2^2} 
+\frac{\ee^3(m-1)}{\pii w m(m+1)}
\frac{(-\log(2\vartheta_3))}{\widetilde{\vartheta}_3^2} 
+\ee^4\frac{(-\log(2\vartheta_4))}{2\pii w\widetilde{\vartheta}_4^2} 
+\frac{(1+w)v(w)}{\sqrt{2\pii}w}C_5\\
C_8
&=\frac{\ee^2}{\pii w}
\frac{(-\log(2\vartheta_2))}{4\widetilde{\vartheta}_2^2} 
\frac{D_2}{4}
+\frac{\ee^3(m-1)}{\pii w m(m+1)}
\frac{(-\log(2\vartheta_3))}{\widetilde{\vartheta}_3^2} 
\frac{D_3}{2}
+\ee^4\frac{(-\log(2\vartheta_4))}{2\pii w\widetilde{\vartheta}_4^2} 
\frac{D_4}{2}\\
&\quad{}+\frac{(1+w)v(w)}{\sqrt{2\pii}w}2\kappa C_5C_6.
\end{align*}
For $x_1,x_2,y_1,y_2\in[0,\infty)$, we have 
$\max\{x_1\gamma(y_1),x_2\gamma(y_2)\}
\leq \max\{x_1,x_2\}\gamma(\max\{y_1,y_2\})$.
From this  and the inequalities \eqref{e2523465}, 
\eqref{e2523466}, and \eqref{e2523467}, 
we derive in the general case $n\geq 2$ that
\begin{align*}
\Delta
\leq\max\Bigl\{C_3\gamma(C_4),C_5C_6\gamma(2\kappa C_6),
C_7\gamma\Bigl(\frac{C_8}{C_7}\Bigr)\Bigr\}
\leq C_1\gamma(C_2),
\end{align*}
where $C_1=\max\{C_3,C_5C_6,C_7\}$ and 
$C_2=\frac{1}{C_1}\max\{C_3C_4,2\kappa C_5C_6^2,C_8\}$. 
Letting $w=0.89$, $m=1367$, $C_4=7.915$, $C_5=0.047$, and $C_6=33$, 
we obtain $v(w)=5.329260\dots$, $C_3=1.2992\dots$, $C_1\leq15.84$, 
$C_2\leq0.65$. Now identify $C_1$ and $C_2$ with their upper bounds.
\end{proof}


\section{A useful identity for permanents}\label{s8257}
The proof of Proposition \ref{p719876} is based on the following 
proposition, which, for future reference, is presented here under 
slightly more general assumptions than needed. It contains an 
identity for the permanent of the matrix 
$(\ee^{y_{j,r}})\in\CC^{n\times n}$, that is,
the term $\sum_{r\in\set{n}_{\neq}^n}\ee^{c_r}$ in \eqref{e617865}
below.  

In what follows, let $n\in\NN\setminus\{1\}$, 
$Y=(y_{j,r})\in\CC^{n\times n}$, and
$z_{j,k,r,s}=y_{j,r}-y_{k,r}-y_{j,s}+y_{k,s}$
for all $j,k,r,s\in\set{n}$, 
\begin{align}\label{e861456}
c_r=\sum_{j\in\set{n}}y_{j,r(j)}
\quad\mbox{for} \quad r\in\set{n}_{\neq}^n, \quad
\alpha=\frac{1}{n}\sum_{j\in\set{n}}\sum_{r\in\set{n}}y_{j,r}, \quad
\beta=\frac{1}{n-1}
\sum_{(j,r)\in\set{n}^2}\widetilde{y}_{j,r}^2,
\end{align}
where, for $(j,r)\in\set{n}^2$, 
\begin{align*}
\widetilde{y}_{j,r}
=y_{j,r}-y_{\bfcdot,r}-y_{j,\bfcdot}+y_{\bfcdot,\bfcdot},\quad
y_{\bfcdot,r}
=\frac{1}{n}\sum_{k=1}^ny_{k,r}, \quad
y_{j,\bfcdot}
=\frac{1}{n}\sum_{s=1}^ny_{j,s}, \quad
y_{\bfcdot,\bfcdot}
=\frac{1}{n^2}\sum_{(k,s)\in\set{n}^2}y_{k,s}.
\end{align*}
\begin{proposition} \label{p874976}
Under the assumptions above, we have 
\begin{align}\label{e617865}
\frac{1}{n!}\sum_{r\in\set{n}_{\neq}^n}\ee^{c_r}
-\ee^{\alpha+\beta/2} 
=\frac{1}{n!}\int_0^1 f(u) \exp\Bigl((1-u)\alpha
+(1-u^2)\frac{\beta}{2}\Bigr)\,\dd u,
\end{align}
where  
\begin{align*}
f(u)&=\frac{f_1(u)}{4n}
+\frac{f_2(u)}{n^2(n-1)}+\frac{f_3(u)}{4n^2(n-1)},\\
f_1(u)
&=\sum_{(j,k)\in\set{n}_{\neq}^2}
\sum_{r\in\set{n}_{\neq}^n}z_{j,k,r(j),r(k)}
(1-uz_{j,k,r(j),r(k)}-\exp(-uz_{j,k,r(j),r(k)})) \ee^{uc_r},\\
f_2(u)
&=\sum_{(j,k,\ell)\in\set{n}_{\neq}^3}
\sum_{r\in\set{n}_{\neq}^n}uz_{j,k,r(j),r(k)}^2
(1-\exp(uz_{j,\ell,r(\ell),r(j)}))\ee^{uc_r},\\
f_3(u)
&=\sum_{(j,k,\ell,m)\in\set{n}_{\neq}^4}\sum_{r\in\set{n}_{\neq}^n}u
z_{j,k,r(j),r(k)}^2
(1-\exp(uz_{j,\ell,r(\ell),r(j)}+uz_{k,m,r(m),r(k)}))\ee^{uc_r}.
\end{align*}
\end{proposition}

We note that $f_2(u)=0$ for $n=2$, and $f_3(u)=0$ for $n\in\{2,3\}$.
Further, we have 
\begin{align}\label{e9726578}
\beta
=\frac{1}{4n^2(n-1)}\sum_{(j,k)\in\set{n}_{\neq}^2}
\sum_{(r,s)\in\set{n}_{\neq}^2}z_{j,k,r,s}^2,
\end{align}
the proof of which is provided in Section \ref{s7145876}. 
For the proof of Proposition \ref{p874976}, we need the next lemma. 
\begin{lemma}\label{l63285}
Let the assumptions of Proposition \ref{p874976} hold. 
For $u\in[0,1]$, we then have 
\begin{align}
\sum_{r\in\set{n}_{\neq}^n}(c_r-\alpha-u\beta)\ee^{uc_r}
&=f(u).
\label{e872860}
\end{align}
\end{lemma}
\begin{proof}[Proof of Proposition \ref{p874976}]
Let $T$ denote the left-hand side of (\ref{e617865}). Then
\begin{align*}
T
&=\frac{1}{n!}\sum_{r\in\set{n}_{\neq}^n}
\exp\Bigl(uc_r+(1-u)\alpha
+(1-u^2)\frac{\beta}{2}\Bigr)\Big|_{u=0}^{u=1}\\
&=\frac{1}{n!}\int_0^1\sum_{r\in\set{n}_{\neq}^n}
(c_r-\alpha-u\beta)\ee^{uc_r}
\exp\Bigl((1-u)\alpha+(1-u^2)\frac{\beta}{2}\Bigr)\,\dd u.
\end{align*}
The assertion now follows by applying Lemma \ref{l63285}.
\end{proof}

In the proof of Lemma \ref{l63285}, we shall apply a technique 
similar to the one used in \citet{Roos2020}. 
We need the following lemma. 
\begin{lemma}\label{l7617655}
Let the assumptions of Proposition \ref{p874976} hold. Further, 
let $g:\,\set{n}_{\neq}^2\longrightarrow\CC$ be a function and 
$(j,k)\in\set{n}_{\neq}^2$. Then
\begin{align}\label{e175376}
\sum_{r\in\set{n}_{\neq}^n}g(r(j),r(k))\ee^{c_r}
=\sum_{r\in\set{n}_{\neq}^n}g(r(k),r(j))\exp(z_{j,k,r(k),r(j)})
\ee^{c_r}.
\end{align}
\end{lemma}
\begin{proof}
The assertion follows by interchanging $r(j)$ with $r(k)$
in \eqref{e175376}. In fact, 
\begin{align*}
\sum_{r\in\set{n}_{\neq}^n}g(r(j),r(k))\ee^{c_r}
&=\sum_{r\in\set{n}_{\neq}^n}g(r(j),r(k))
\exp(y_{j,r(j)}+y_{k,r(k)})
\exp\Bigl(\sum_{\ell\in\set{n}\setminus\{j,k\}}y_{\ell,r(\ell)}
\Bigr)\\
&=\sum_{r\in\set{n}_{\neq}^n}g(r(k),r(j))
\exp(y_{j,r(k)}+y_{k,r(j)})
\exp\Bigl(\sum_{\ell\in\set{n}\setminus\{j,k\}}y_{\ell,r(\ell)}
\Bigr),
\end{align*}
which is equal to the right-hand side of \eqref{e175376}. 
\end{proof}

\begin{proof}[Proof of Lemma \ref{l63285}]
By considering $uy_{j,r}$ instead of $y_{j,r}$ for 
$(j,r)\in\set{n}^2$,
without loss of generality, we may assume that $u=1$. 
Let $T_j=f_j(1)$ for $j\in\set{3}$. 
Let $T$ denote the left-hand side of 
(\ref{e872860}), i.e.\ $T=T_4-T_5$, where 
\begin{align*}
T_4=\sum_{r\in\set{n}_{\neq}^n}(c_r-\alpha)\ee^{c_r}, \qquad 
T_5=\sum_{r\in\set{n}_{\neq}^n}\beta\ee^{c_r}.
\end{align*}
For $r\in\set{n}_{\neq}^n$, we have 
$\alpha=\frac{1}{n}\sum_{j\in\set{n}}\sum_{k\in\set{n}}y_{j,r(k)}$ 
and hence, using Lemma \ref{l7617655}, 
\begin{align}
T_{4}
&=\sum_{j\in\set{n}}\sum_{r\in\set{n}_{\neq}^n}
\Bigl( y_{j,r(j)}-\frac{1}{n}\sum_{k\in\set{n}}y_{j,r(k)}\Bigr)
\ee^{c_r}
=\frac{1}{n}\sum_{(j,k)\in\set{n}_{\neq}^2}
\sum_{r\in\set{n}_{\neq}^n}(y_{j,r(j)}-y_{j,r(k)})\ee^{c_r}
\label{e77524365}\\
&=-\frac{1}{n}\sum_{(j,k)\in\set{n}_{\neq}^2}
\sum_{r\in\set{n}_{\neq}^n}(y_{j,r(j)}-y_{j,r(k)})
\exp(z_{j,k,r(k),r(j)})\ee^{c_r}.\label{e8625}
\end{align}
By adding the right-hand sides in \eqref{e77524365} 
and \eqref{e8625} and dividing by two, we get 
\begin{align}
T_{4}
&=\frac{1}{2n}\sum_{(j,k)\in\set{n}_{\neq}^2}
\sum_{r\in\set{n}_{\neq}^n}(y_{j,r(j)}-y_{j,r(k)})
(1-\exp(z_{j,k,r(k),r(j)}))\ee^{c_r}\label{e837696}\\
&=\frac{1}{2n}\sum_{(j,k)\in\set{n}_{\neq}^2}
\sum_{r\in\set{n}_{\neq}^n}(y_{k,r(k)}-y_{k,r(j)})
(1-\exp(z_{j,k,r(k),r(j)}))\ee^{c_r},\label{e837697}
\end{align}
where the last equality follows by interchanging $j$ with $k$
and using the fact that $z_{j,k,r,s}=-z_{k,j,r,s}=-z_{j,k,s,r}$
for all $j,k,r,s\in\set{n}$. 
Adding the right-hand sides of (\ref{e837696}) and (\ref{e837697})
and dividing by two, we obtain
\begin{align*}
T_{4}
&=\frac{1}{4n}\sum_{(j,k)\in\set{n}_{\neq}^2}
\sum_{r\in\set{n}_{\neq}^n}z_{j,k,r(j),r(k)}
(1-\exp(z_{j,k,r(k),r(j)}))\ee^{c_r}
=\frac{T_1}{4n}+T_{6},
\end{align*} 
where
\begin{align*}
T_{6}
&=\frac{1}{4n}\sum_{(j,k)\in\set{n}_{\neq}^2}
\sum_{r\in\set{n}_{\neq}^n}z_{j,k,r(j),r(k)}^2\ee^{c_r}. 
\end{align*}
Therefore $T=T_4-T_5=\frac{T_1}{4n}+T_6-T_5$. 
From \eqref{e9726578}, it follows that, for $r\in\set{n}_{\neq}^n$, 
\begin{align*}
4n^2(n-1)\beta
=\sum_{(j,k)\in\set{n}_{\neq}^2}
\sum_{(\ell,m)\in\set{n}_{\neq}^2}z_{j,k,r(\ell),r(m)}^2.
\end{align*}
This gives
\begin{align*}
4n^2(n-1)(T_6-T_5) 
&=n(n-1)\sum_{(j,k)\in\set{n}_{\neq}^2}
\sum_{r\in\set{n}_{\neq}^n}\Bigl(z_{j,k,r(j),r(k)}^2
-\frac{1}{n(n-1)}\sum_{(\ell,m)\in\set{n}_{\neq}^2}
z_{j,k,r(\ell),r(m)}^2\Bigr)\ee^{c_r}\\
&=\sum_{(j,k)\in\set{n}_{\neq}^2}\sum_{(\ell,m)\in\set{n}_{\neq}^2}
\sum_{r\in\set{n}_{\neq}^n}(z_{j,k,r(j),r(k)}^2
-z_{j,k,r(\ell),r(m)}^2)\ee^{c_r}.
\end{align*}
Considering the different cases of $\ell,m\in\{j,k\}$ or not, we 
obtain
\begin{align}
4n^2(n-1)(T_6-T_5) 
&=\sum_{(j,k)\in\set{n}_{\neq}^2}
\sum_{\ell=j, m\in\set{n}\setminus\{j,k\}}
\sum_{r\in\set{n}_{\neq}^n}(z_{j,k,r(j),r(k)}^2
-z_{j,k,r(j),r(m)}^2)\ee^{c_r}\nonumber\\
&\quad{}+\sum_{(j,k)\in\set{n}_{\neq}^2}
\sum_{\ell=k,m\in\set{n}\setminus\{j,k\}}
\sum_{r\in\set{n}_{\neq}^n}(z_{j,k,r(j),r(k)}^2
-z_{j,k,r(k),r(m)}^2)\ee^{c_r}\nonumber\\
&\quad{}+\sum_{(j,k)\in\set{n}_{\neq}^2}
\sum_{\ell\in\set{n}\setminus\{j,k\},m=j}
\sum_{r\in\set{n}_{\neq}^n}(z_{j,k,r(j),r(k)}^2
-z_{j,k,r(\ell),r(j)}^2)\ee^{c_r}\nonumber\\
&\quad{}+\sum_{(j,k)\in\set{n}_{\neq}^2}
\sum_{\ell\in\set{n}\setminus\{j,k\},m=k}
\sum_{r\in\set{n}_{\neq}^n}(z_{j,k,r(j),r(k)}^2
-z_{j,k,r(\ell),r(k)}^2)
\ee^{c_r}\nonumber\\
&\quad{}+\sum_{(j,k)\in\set{n}_{\neq}^2}
\sum_{(\ell,m)\in(\set{n}\setminus\{j,k\})_{\neq}^2}
\sum_{r\in\set{n}_{\neq}^n}(z_{j,k,r(j),r(k)}^2
-z_{j,k,r(\ell),r(m)}^2)\ee^{c_r}.\label{e75215}
\end{align}
We note that, $z^2_{j,k,r,s}=z^2_{k,j,r,s}=z^2_{j,k,s,r}$ for all
$j,k,r,s\in\set{n}$, so that, in the case $\{\ell,m\}=\{j,k\}$, the 
respective summand is zero. From \eqref{e75215}, we obtain that
$4n^2(n-1)(T_6-T_5)=T_7+T_8$, 
where 
\begin{align*}
T_{7}
&=4\sum_{(j,k,\ell)\in\set{n}_{\neq}^3}
\sum_{r\in\set{n}_{\neq}^n}(z_{j,k,r(j),r(k)}^2
-z_{j,k,r(j),r(\ell)}^2)\ee^{c_r}, \\
T_{8}
&=\sum_{(j,k,\ell,m)\in\set{n}_{\neq}^4}
\sum_{r\in\set{n}_{\neq}^n}(z_{j,k,r(j),r(k)}^2
-z_{j,k,r(\ell),r(m)}^2)\ee^{c_r}. 
\end{align*}
Using Lemma \ref{l7617655}, 
we find that 
\begin{align*}
T_7
&=4\sum_{(j,k,\ell)\in\set{n}_{\neq}^3}
\sum_{r\in\set{n}_{\neq}^n}(z_{j,k,r(j),r(k)}^2
-z_{j,k,r(j),r(k)}^2\exp(z_{k,\ell,r(\ell),r(k)}))\ee^{c_r}
=4T_2, \\
T_{8}
&=\sum_{(j,k,\ell,m)\in\set{n}_{\neq}^4}\sum_{r\in\set{n}_{\neq}^n}
(z_{j,k,r(j),r(k)}^2-z_{j,k,r(j),r(k)}^2
\exp(z_{j,\ell,r(\ell),r(j)}+z_{k,m,r(m),r(k)}))\ee^{c_r}
=T_3.
\end{align*}
Combining the identities above, we derive 
$T
=\frac{T_1}{4n}+\frac{1}{4n^2(n-1)}(T_7+T_8)
=\frac{T_1}{4n}
+\frac{1}{4n^2(n-1)}(4T_2+T_3)$, which implies the assertion. 
\end{proof}
\section{Auxiliary results and the proof of Proposition \ref{p719876}}
\label{s5254}
In what follows, let the assumptions of Section \ref{s86256} hold.
\begin{lemma} 
For $x\in\RR$ and $k\in\Zpl=\{0,1,2,\dots\}$, we have 
\begin{gather}
\ABS{\ee^{\ii x}-\sum_{j=0}^k\frac{(\ii x)^j}{j!}}
\leq2\frac{\abs{x}^k}{k!}
\min\Bigl\{1,\frac{\abs{x}}{2(k+1)}\Bigr\}.\label{e982759}
\end{gather}
\end{lemma}
The proof of a more general assertion can be found in 
\citet[Lemma 1, p.~295]{MR1476912}. The following lemma is shown in 
\citet{Roos2019}. It complements Theorem 2.1 in \citet{MR719749}
under the present assumptions.  
Unlike the upper bound given in that theorem, 
our bound contains explicit constants and is valid for all $t\in\RR$.
Here, for $x\in\RR$, let $\floor{x}$ be the largest integer $\leq x$. 
\begin{lemma}\label{l765245}
Let $n\in\NN\setminus\{1\}$ and $t\in\RR$. Then  
\begin{align*}
\abs{\varphi(t)}
&=\frac{1}{n!}\ABS{\sum_{r\in\set{n}_{\neq}^n}
\exp\Bigl(\ii t\sum_{j\in\set{n}}a_{j,r(j)}\Bigr)}
\leq \Bigl(\frac{1}{n^2(n-1)^2}
\sum_{(j,k)\in \set{n}_{\neq}^2}\sum_{(r,s)\in \set{n}_{\neq}^2}
\cos^2\Bigl(\frac{t b_{j,k,r,s}}{2}\Bigr)\Bigr)^{\floor{n/2}/2}.
\end{align*}
\end{lemma}
\begin{lemma}\label{l287659}
Let $L,M\subseteq\set{n}$ with $\card{L}=\card{M}=:\ell$ and let
$t\in\RR$. Then
\begin{align}\label{e9756}
\frac{1}{(n-\ell)!}\ABS{\sum_{r\in(\set{n}\setminus L)_{\neq}^{
\set{n}\setminus M}}\exp\Bigl(\ii t\sum_{j\in\set{n}\setminus M}
a_{j,r(j)}\Bigr)}
\leq h_\ell(t),
\end{align}
where $\sigma^2$, $\gamma(x)$ for $x\in\RR$, $\kappa$ and 
$h_{\ell}(t)$ are given as in \eqref{e872897}, \eqref{e78256}, 
\eqref{e982758}, and \eqref{e29876898}. 
\end{lemma}
\begin{proof}
Let $T$ denote the left-hand side of \eqref{e9756}. 
If $\ell\in\{n-1,n\}$, then $T=1=h_\ell(t)$ and hence 
\eqref{e9756} is valid. 
Let us now assume that $\ell\leq n-2$.
Let $p:\,\set{n-\ell}\longrightarrow\set{n}\setminus M$ and
$q:\,\set{n-\ell}\longrightarrow\set{n}\setminus L$ be two 
bijective maps. Clearly, we have 
\begin{align*}
T&=\frac{1}{(n-\ell)!}
\ABS{\sum_{r\in(\set{n-\ell})_{\neq}^{\set{n-\ell}}}
\exp\Bigl(\ii t\sum_{j\in\set{n-\ell}}a_{p(j),q(r(j))}\Bigr)}.
\end{align*}
Lemma~\ref{l765245} together with the inequality 
$\floor{\frac{x}{2}}\geq\frac{x-1}{2}$ for all integers $x$ gives 
\begin{align}
T
&\leq\Bigl(\frac{1}{(n-\ell)^2(n-\ell-1)^2}
\sum_{(j,k)\in(\set{n}\setminus M)_{\neq}^2}
\sum_{(r,s)\in(\set{n}\setminus L)_{\neq}^2}
\cos^2\Bigl(\frac{tb_{j,k,r,s}}{2}\Bigr)\Bigr)^{\floor{(n-\ell)/2}/2}
\nonumber\\
&\leq\Bigl(\frac{n^2(n-1)^2}{(n-\ell)^2(n-\ell-1)^2}
\frac{1}{n^2(n-1)^2}
\sum_{(j,k)\in\set{n}_{\neq}^2}\sum_{(r,s)\in\set{n}_{\neq}^2}
\cos^2\Bigl(\frac{tb_{j,k,r,s}}{2}\Bigr)\Bigr)^{(n-\ell-1)/4}.
\label{e45165}
\end{align}
Using the inequality $x\leq \exp(\frac{x-1}{\sqrt{x}})$ for 
$x\in[1,\infty)$ (see \citet[page 272, 3.6.15]{MR0274686}), 
we obtain that 
$\frac{n}{n-\ell}
\leq \exp(\frac{\ell}{\sqrt{n(n-\ell)}})$
and 
$\frac{n-1}{n-\ell-1}
\leq \exp(\frac{\ell}{\sqrt{(n-1)(n-\ell-1)}})$.
Therefore
\begin{align}
\Bigl(\frac{n(n-1)}{(n-\ell)(n-\ell-1)}\Bigr)^{(n-\ell-1)/2}
&\leq \exp\Bigl(\frac{n-\ell-1}{2}\Bigl(\frac{\ell}{\sqrt{n(n-\ell)}}+
\frac{\ell}{\sqrt{(n-1)(n-\ell-1)}}\Bigr)\Bigr)\nonumber\\
&\leq\exp\Bigl(\frac{\ell}{2}\Bigl(\sqrt{\frac{n-\ell-1}{n}}+
\sqrt{\frac{n-\ell-1}{n-1}}\Bigr)\Bigr)
\leq \ee^\ell.\label{e67812458} 
\end{align}
Using \eqref{e45165}, \eqref{e67812458}, the inequality 
$1+x\leq \ee^x$ for $x\in\RR$, the identity 
$\cos^2(x)=\frac{1}{2}(1+\cos(2x))$ for $x\in\RR$, 
and \eqref{e982758}, we get
\begin{align*}
T
&\leq\ee^\ell\Bigl(1+\frac{1}{n^2(n-1)^2}
\sum_{(j,k)\in \set{n}_{\neq}^2}\sum_{(r,s)\in \set{n}_{\neq}^2}
\Bigl(\cos^2\Bigl(\frac{tb_{j,k,r,s}}{2}\Bigr)-1\Bigr)
\Bigr)^{(n-\ell-1)/4}\\
&\leq \ee^\ell\exp\Bigl(\frac{n-\ell-1}{8n^2(n-1)^2}
\sum_{(j,k)\in \set{n}_{\neq}^2}\sum_{(r,s)\in \set{n}_{\neq}^2}
(\cos(tb_{j,k,r,s})-1)\Bigr)\\
&\leq \ee^\ell\exp\Bigl(\frac{n-\ell-1}{16n^2(n-1)^2}
\sum_{(j,k)\in \set{n}_{\neq}^2}\sum_{(r,s)\in \set{n}_{\neq}^2}
\bigl(-(tb_{j,k,r,s})^2+(tb_{j,k,r,s})^2
\min\{1,2\kappa\abs{tb_{j,k,r,s}}\}\bigr)\Bigr)\\
&=\exp\Bigl(\ell-\frac{n-\ell-1}{4(n-1)}t^2
\Bigl(\sigma^2-\frac{1}{4}\gamma(2\kappa t)\Bigr)\Bigr),
\end{align*}
which, together with the obvious relations $T\leq 1$
and $\bbone_{[\ell,\infty)}(n)=1$, implies the assertion. 
\end{proof}

\begin{lemma}\label{l375876}
For $x,y\in\RR$, let $g(x,y)=x^2\min\{1,\abs{y}\}$. Then,
for $c,x,y,z\in\RR$, we have
\begin{enumerate}[(a)]

\item \label{l375876.a}
$g(x,y+z)
\leq g(x,y)+g(x,z)\leq 2 g(x,\frac{\abs{y}+\abs{z}}{2})$,

\item \label{l375876.b} 
$g(x,cy)+g(y,cx)\leq g(x,cx)+g(y,cy)$,

\item \label{l375876.c} 
$x^2-\frac{4}{27c^2}\leq g(x,cx)$, if $c\neq0$.
\end{enumerate}
\end{lemma}
\begin{proof}
\begin{enumerate}[(a)]

\item This is clear. 

\item This is shown by using that
\begin{align*}
g(x,cx)+g(y,cy)-g(x,cy)-g(y,cx)
=(x^2-y^2)(\min\{1,\abs{cx}\}-\min\{1,\abs{cy}\})\geq0.
\end{align*}

\item Let $c\neq0$ and $y\in[0,\infty)$. Then
\begin{align*}
\frac{4}{27c^2}-y+\abs{c}y^{3/2}
=\frac{1}{27c^2}(3\abs{c}\sqrt{y}+1)(3\abs{c}\sqrt{y}-2)^2
\geq0. 
\end{align*}
For $x\in\RR$ and $y=x^2\bbone_{[0,1]}(\abs{cx})$, this implies 
\begin{align*}
x^2-\frac{4}{27c^2}
&\leq x^2-y+\abs{c}y^{3/2}
=x^2\bbone_{(1,\infty)}(\abs{cx})+\abs{c}\abs{x}^3
\bbone_{[0,1]}(\abs{cx})
=g(x,cx).\qedhere
\end{align*}
\end{enumerate}
\end{proof}

\begin{proof}[Proof of Proposition \ref{p719876}]
We shall use Proposition \ref{p874976} for the matrix $Y=(y_{j,r})$
with $y_{j,r}=\ii ta_{j,r}$ for $j,r\in\set{n}$. 
The quantities $z_{j,k,r,s}$, $c_r$, $\alpha$, and  $\beta$
used there are given by
$z_{j,k,r,s}
=\ii t b_{j,k,r,s}$ for all 
$j,k,r,s\in\set{n}$ and
\begin{align*} 
c_r
&=\ii td_r \quad\mbox{ with }
d_r
=\sum_{j\in\set{n}}a_{j,r(j)}
\quad\mbox{for }r\in\set{n}_{\neq}^n,\quad 
\alpha 
=\ii t\mu,\quad
\beta
=-\sigma^2t^2.
\end{align*}
Hence
\begin{align*}
\abs{\varphi(t)-\ee^{\ii t\mu-\sigma^2t^2/2}}
&=\ABS{\frac{1}{n!}\sum_{r\in\set{n}_{\neq}^n}\ee^{c_r}
-\ee^{\alpha+\beta/2}} 
\leq \frac{1}{n!}\int_0^1 \abs{f(u)} 
\exp\Bigl(-(1-u^2)\frac{\sigma^2t^2}{2}\Bigr)\,\dd u,
\end{align*}
where  
\begin{align*}
f(u)&=\frac{f_1(u)}{4n}
+\frac{f_2(u)}{n^2(n-1)}+\frac{f_3(u)}{4n^2(n-1)},\\
f_1(u)
&=\sum_{(j,k)\in\set{n}_{\neq}^2}
\sum_{r\in\set{n}_{\neq}^n}\ii t b_{j,k,r(j),r(k)}
(1-\ii tu b_{j,k,r(j),r(k)}
-\exp(-\ii tu b_{j,k,r(j),r(k)}))\ee^{\ii t ud_r},\\
f_2(u)
&=\sum_{(j,k,\ell)\in\set{n}_{\neq}^3}
\sum_{r\in\set{n}_{\neq}^n}t^2u b_{j,k,r(j),r(k)}^2
(\exp(\ii tu b_{j,\ell,r(\ell),r(j)})-1)\ee^{\ii t ud_r},\\
f_3(u)
&=\sum_{(j,k,\ell,m)\in\set{n}_{\neq}^4}\sum_{r\in\set{n}_{\neq}^n}
t^2ub_{j,k,r(j),r(k)}^2
(\exp(\ii tu(b_{j,\ell,r(\ell),r(j)}+b_{k,m,r(m),r(k)}))-1)
\ee^{\ii t ud_r}.
\end{align*}
The proof is completed with the help of the following lemma.
\end{proof}
\begin{lemma} 
Let the assumptions in the proof of Proposition \ref{p719876} hold. 
For $u\in[0,1]$, we have
\begin{align}
\frac{\abs{f_1(u)}}{n!4n}
&\leq \frac{t^2u}{2}h_2(tu)\gamma\Bigl(\frac{tu}{4}\Bigr),
\label{e721456}\\
\frac{\abs{f_2(u)}}{n!n^2(n-1)}
&\leq \frac{2(n-2)t^2u}{n(n-1)}h_3(tu)\gamma\Bigl(\frac{tu}{2}\Bigr),
\label{e78252}\\
\frac{\abs{f_3(u)}}{n!\,4n^2(n-1)}
&\leq \frac{(n-2)(n-3)t^2u}{n(n-1)}
h_4(tu)\gamma\Bigl(\frac{tu}{2}\Bigr).
\label{e8976218}
\end{align}
\end{lemma}

\begin{proof}[Proof of \eqref{e721456}]
For $u\in[0,1]$, $(j,k)\in\set{n}_{\neq}^2$, and 
$(v,w)\in\set{n}_{\neq}^2$, inequality \eqref{e982759} implies that 
\begin{align*}
\abs{1-\ii tu  b_{j,k,v,w}
-\exp(-\ii tu  b_{j,k,v,w})}
\leq 2\abs{tub_{j,k,v,w}}
\min\Bigl\{1,\frac{\abs{tu}}{4}\abs{b_{j,k,v,w}}\Bigr\}
\end{align*} 
and Lemma \ref{l287659} gives 
\begin{align*}
\frac{1}{(n-2)!}
\ABS{\sum_{\newatop{r\in\set{n}_{\neq}^n}{r(j)=v,r(k)=w}}
\ee^{\ii t u  d_r}}
&=\frac{1}{(n-2)!}
\ABS{\sum_{r\in(\set{n}\setminus\{v,w\})_{\neq}^{\set{n}
\setminus\{j,k\}}}
\exp\Bigl(\ii tu\sum_{\ell\in\set{n}\setminus\{j,k\}}
a_{\ell,r(\ell)}\Bigr)} 
\leq h_2(tu).
\end{align*}
Therefore
\begin{align*}
\frac{\abs{f_1(u)}}{n!4n}
&\leq\frac{1}{n!4n}
\sum_{(j,k)\in\set{n}_{\neq}^2}
\sum_{(v,w)\in\set{n}_{\neq}^2}
\abs{t b_{j,k,v,w}}
\abs{1-\ii tub_{j,k,v,w}-\exp(-\ii t u b_{j,k,v,w})}
\ABS{\sum_{\newatop{r\in\set{n}_{\neq}^n}{r(j)=v,r(k)=w}}
\ee^{\ii t ud_r}}\\ 
&\leq\frac{t^2u}{2}h_2(tu)\frac{1}{n^2(n-1)}
\sum_{(j,k)\in\set{n}_{\neq}^2}
\sum_{(v,w)\in\set{n}_{\neq}^2}b_{j,k,v,w}^2
\min\Bigl\{1,\frac{\abs{tu}}{4}\abs{b_{j,k,v,w}}\Bigr\}\\
&=\frac{t^2u}{2}h_2(tu)\gamma\Bigl(\frac{tu}{4}\Bigr).
\end{align*}
which shows \eqref{e721456}. 
\end{proof}

\begin{proof}[Proof of \eqref{e78252}]
If $n=2$, then both sides of the inequality in \eqref{e78252} 
are equal to zero. Let us now assume that $n\geq 3$.  
For $u\in[0,1]$,  $(j,k,\ell)\in\set{n}_{\neq}^3$, and 
$(p,v,w)\in\set{n}_{\neq}^3$, 
we obtain from \eqref{e982759} that
\begin{align*}
\abs{\exp(\ii t u b_{j,\ell,w,p})-1}
\leq 2\min\Bigl\{1, \frac{\abs{tu}}{2}\abs{b_{j,\ell,w,p}}\Bigr\}
\end{align*}
and Lemma \ref{l287659} implies that 
\begin{align*}
\frac{1}{(n-3)!}
\ABS{\sum_{\newatop{r\in\set{n}_{\neq}^n}{(r(j),r(k),r(\ell))=(p,v,w)}}
\ee^{\ii tu d_r}}
&=\frac{1}{(n-3)!}
\ABS{\sum_{r\in(\set{n}\setminus\{p,v,w\})_{\neq}^{\set{n}
\setminus\{j,k,\ell\}}}\exp\Bigl(\ii tu
\sum_{m\in\set{n}\setminus\{j,k,\ell\}}a_{m,r(m)}\Bigr)}\\
&\leq h_3(tu).
\end{align*}
Hence
\begin{align*}
\frac{\abs{f_2(u)}}{n!\,n^2(n-1)}
&\leq \frac{1}{n!\,n^2(n-1)}
\sum_{(j,k,\ell)\in\set{n}_{\neq}^3}
\sum_{(p,v,w)\in\set{n}_{\neq}^3}t^2u b_{j,k,p,v}^2\\
&\quad{}\times
\abs{\exp(\ii tu b_{j,\ell,w,p})-1}
\ABS{\sum_{\newatop{r\in\set{n}_{\neq}^n}{
(r(j),r(k),r(\ell))=(p,v,w)}}\ee^{\ii t ud_r}}\\
&\leq\frac{2t^2u}{n^3(n-1)^2(n-2)}h_3(tu)T(u),
\end{align*}
where 
\begin{align*}
T(u)&=\sum_{(j,k,\ell)\in\set{n}_{\neq}^3}
\sum_{(p,v,w)\in\set{n}_{\neq}^3} b_{j,k,p,v}^2
\min\Bigl\{1, \frac{\abs{tu}}{2}\abs{b_{j,\ell,p,w}}\Bigr\}.
\end{align*}
Here we used that $\abs{b_{j,\ell,w,p}}=\abs{b_{j,\ell,p,w}}$.
By using $g$ as in Lemma \ref{l375876}, we get 
\begin{align*}
T(u)
&=\sum_{(j,k,\ell)\in\set{n}_{\neq}^3}
\sum_{(p,v,w)\in\set{n}_{\neq}^3}
g\Bigl(b_{j,k,p,v},\frac{tu}{2}b_{j,\ell,p,w}\Bigr)
=\sum_{(j,k,\ell)\in\set{n}_{\neq}^3}
\sum_{(p,v,w)\in\set{n}_{\neq}^3}
g\Bigl(b_{j,\ell,p,w},\frac{tu}{2}b_{j,k,p,v}\Bigr),
\end{align*}
where the latter equality follows by interchanging $k$ 
with $\ell$ and $v$ with $w$. Using 
Lemma \ref{l375876}\ref{l375876.b}, 
\begin{align*}
T(u)
&=\sum_{(j,k,\ell)\in\set{n}_{\neq}^3}
\sum_{(p,v,w)\in\set{n}_{\neq}^3}
\frac{1}{2}\Bigl(g\Bigl(b_{j,k,p,v},\frac{tu}{2}b_{j,\ell,p,w}\Bigr)
+g\Bigl(b_{j,\ell,p,w},\frac{tu}{2}b_{j,k,p,v}\Bigr)\Bigr)\\
&\leq\sum_{(j,k,\ell)\in\set{n}_{\neq}^3}
\sum_{(p,v,w)\in\set{n}_{\neq}^3}\frac{1}{2}\Bigl(
g\Bigl(b_{j,k,p,v},\frac{tu}{2}b_{j,k,p,v}\Bigr)
+g\Bigl(b_{j,\ell,p,w},\frac{tu}{2}b_{j,\ell,p,w}\Bigr)\Bigr)\\
&=\sum_{(j,k,\ell)\in\set{n}_{\neq}^3}
\sum_{(p,v,w)\in\set{n}_{\neq}^3} 
g\Bigl(b_{j,k,p,v},\frac{tu}{2}b_{j,k,p,v}\Bigr)
=n^2(n-1)(n-2)^2\gamma\Bigl(\frac{tu}{2}\Bigr).
\end{align*}
Consequently 
\begin{align*} 
\frac{\abs{f_2(u)}}{n!\,n^2(n-1)}
&\leq \frac{2t^2u}{n^3(n-1)^2(n-2)}h_3(tu)
n^2(n-1)(n-2)^2\gamma\Bigl(\frac{tu}{2}\Bigr)\\
&= \frac{2(n-2)t^2u}{n(n-1)}h_3(tu)\gamma\Bigl(\frac{tu}{2}\Bigr),
\end{align*}
which proves \eqref{e78252}. 
\end{proof}

\begin{proof}[Proof of \eqref{e8976218}]
If $n\in\{2,3\}$, then both sides of the inequality in 
\eqref{e8976218} are equal to zero. Let us now assume that $n\geq4$. 
For $u\in[0,1]$, $(j,k,\ell,m)\in\set{n}_{\neq}^4$, and 
$(p,q,v,w)\in\set{n}_{\neq}^4$, inequality \eqref{e982759} gives
\begin{align*}
\abs{\exp(\ii tu(b_{j,\ell,v,p}+b_{k,m,w,q}))-1}
\leq2\min\Bigl\{1,\frac{\abs{tu}}{2}\abs{b_{j,\ell,v,p}+b_{k,m,w,q}}
\Bigr\}
\end{align*}
and from Lemma \ref{l287659} it follows that
\begin{align*}
\lefteqn{\frac{1}{(n-4)!}
\ABS{\sum_{\newatop{r\in\set{n}_{\neq}^n}{
(r(j),r(k),r(\ell),r(m))=(p,q,v,w)}}\ee^{\ii t u d_r}}}\\
&=\frac{1}{(n-4)!}
\ABS{\sum_{r\in(\set{n}\setminus\{p,q,v,w\})_{\neq}^{\set{n}
\setminus\{j,k,\ell,m\}}}\exp\Bigl(\ii t u 
\sum_{s\in\set{n}\setminus\{j,k,\ell,m\}}a_{s,r(s)}\Bigr)}
\leq h_4(tu).
\end{align*}
Hence 
\begin{align}
\frac{\abs{f_3(u)}}{n!4n^2(n-1)} 
&\leq  \frac{1}{n!4n^2(n-1)}\sum_{(j,k,\ell,m)\in\set{n}_{\neq}^4}
\sum_{(p,q,v,w)\in\set{n}_{\neq}^4}t^2ub_{j,k,p,q}^2\nonumber\\
&\quad{}\times\abs{\exp(\ii tu(b_{j,\ell,v,p}+b_{k,m,w,q}))-1}
\ABS{\sum_{\newatop{r\in\set{n}_{\neq}^n}{
(r(j),r(k),r(\ell),r(m))=(p,q,v,w)}}\ee^{\ii t ud_r}}\nonumber\\
&\leq \frac{(n-4)!t^2u}{n!2n^2(n-1)}h_4(tu) T_1(u),\label{e987326589}
\end{align}
where
\begin{align*}
T_1(u)
&=\sum_{(j,k,\ell,m)\in\set{n}_{\neq}^4}
\sum_{(p,q,v,w)\in\set{n}_{\neq}^4}b_{j,k,p,q}^2
\min\Bigl\{1,\frac{\abs{tu}}{2}\abs{b_{j,\ell,v,p}
+b_{k,m,w,q}}\Bigr\}.
\end{align*}
Using Lemma \ref{l375876}\ref{l375876.a} 
with $g$ being defined there together with 
$\abs{b_{j,k,\ell,m}}=\abs{b_{k,j,\ell,m}}=\abs{b_{j,k,m,\ell}}$ 
for all $j,k,\ell,m\in\set{n}$, it follows that
\begin{align}
T_1(u)
&\leq T_2(u)+T_3(u)\label{e514132}
\end{align}
with 
\begin{align}
T_2(u)
&=\sum_{(j,k,\ell,m)\in\set{n}_{\neq}^4}
\sum_{(p,q,v,w)\in\set{n}_{\neq}^4}
g\Bigl(b_{j,k,p,q},\frac{tu}{2}b_{j,\ell,p,v}\Bigr),\label{e129875}\\
\quad
T_3(u)
&=\sum_{(j,k,\ell,m)\in\set{n}_{\neq}^4}
\sum_{(p,q,v,w)\in\set{n}_{\neq}^4}
g\Bigl(b_{j,k,p,q},\frac{tu}{2}b_{m,k,w,q}\Bigr).
\label{e86214565}
\end{align}
By interchanging $k$ with $\ell$ and $q$ with $v$, we get
\begin{align}
T_2(u)
=\sum_{(j,k,\ell,m)\in\set{n}_{\neq}^4}
\sum_{(p,q,v,w)\in\set{n}_{\neq}^4}
g\Bigl(b_{j,\ell,p,v},\frac{tu}{2}b_{j,k,p,q}\Bigr).
\label{e53165090}
\end{align}
Interchanging $j$ with $m$ and $p$ with $w$ leads to
\begin{align}
T_3(u)
&=\sum_{(j,k,\ell,m)\in\set{n}_{\neq}^4}
\sum_{(p,q,v,w)\in\set{n}_{\neq}^4}
g\Bigl(b_{m,k,w,q},\frac{tu}{2}b_{j,k,p,q}\Bigr).
\label{e75186}
\end{align}
Therefore, by using \eqref{e129875}, \eqref{e53165090}, 
and Lemma \ref{l375876}\ref{l375876.b}, 
\begin{align}
T_2(u)
&=\sum_{(j,k,\ell,m)\in\set{n}_{\neq}^4}
\sum_{(p,q,v,w)\in\set{n}_{\neq}^4}
\frac{1}{2}\Bigl(
g\Bigl(b_{j,k,p,q},\frac{tu}{2}b_{j,\ell,p,v}\Bigr)
+g\Bigl(b_{j,\ell,p,v},\frac{tu}{2}b_{j,k,p,q}\Bigr)\Bigr)
\nonumber\\
&\leq \sum_{(j,k,\ell,m)\in\set{n}_{\neq}^4}
\sum_{(p,q,v,w)\in\set{n}_{\neq}^4}
\frac{1}{2}\Bigl(
g\Bigl(b_{j,k,p,q},\frac{tu}{2}b_{j,k,p,q}\Bigr)
+g\Bigl(b_{j,\ell,p,v},\frac{tu}{2}b_{j,\ell,p,v}\Bigr)\Bigr)
\nonumber\\
&=\sum_{(j,k,\ell,m)\in\set{n}_{\neq}^4}
\sum_{(p,q,v,w)\in\set{n}_{\neq}^4}
g\Bigl(b_{j,k,p,q},\frac{tu}{2}b_{j,k,p,q}\Bigr)
\nonumber\\
&=n^2(n-1)(n-2)^2(n-3)^2\gamma\Bigl(\frac{tu}{2}\Bigr).
\label{e2866598}
\end{align}
and similarly, using \eqref{e86214565}, \eqref{e75186}, and 
Lemma \ref{l375876}\ref{l375876.b}, 
\begin{align}
T_3(u) 
\leq n^2(n-1)(n-2)^2(n-3)^2\gamma\Bigl(\frac{tu}{2}\Bigr).
\label{e71486}
\end{align}
Using \eqref{e514132}, 
\eqref{e2866598}, and \eqref{e71486}, we obtain
\begin{align}
T_1(u)
&\leq 2n^2(n-1)(n-2)^2(n-3)^2\gamma\Bigl(\frac{tu}{2}\Bigr).
\label{e872587}
\end{align}
Combining \eqref{e987326589} and \eqref{e872587}, we get
\eqref{e8976218}.
\end{proof}

\section{Remaining proofs}\label{s7145876}
\begin{proof}[Proof of \eqref{e9726578}]
We have
$\sum_{j=1}^n \widetilde{y}_{j,s}
=0=\sum_{r=1}^n\widetilde{y}_{k,r}$ for all $(k,s)\in\set{n}^2$
and therefore
\eqref{e861456} implies that 
\begin{align}
(n-1)\beta
&=\sum_{j=1}^n\sum_{r=1}^n \widetilde{y}_{j,r}^2
=\frac{1}{n^2}\sum_{j=1}^n\sum_{r=1}^n\sum_{k=1}^n\sum_{s=1}^n
\widetilde{y}_{j,r}(y_{j,r}-y_{k,r}-y_{j,s}+y_{k,s})
=\sum_{j=1}^n\sum_{r=1}^n\widetilde{y}_{j,r}y_{j,r}\nonumber\\
&=\frac{1}{n^2}\sum_{(j,k)\in\set{n}_{\neq}^2}
\sum_{(r,s)\in\set{n}_{\neq}^2}y_{j,r}z_{j,k,r,s}.\label{e7618670}
\end{align}
By interchanging $j$ with $k$ in the right-hand side of 
\eqref{e7618670}, we obtain 
\begin{align}\label{e7618671}
(n-1)\beta
&=-\frac{1}{n^2}
\sum_{(j,k)\in\set{n}_{\neq}^2}\sum_{(r,s)\in\set{n}_{\neq}^2}
y_{k,r}z_{j,k,r,s}.
\end{align}
By adding the right-hand sides of \eqref{e7618670} 
and \eqref{e7618671}, we get
\begin{align}\label{e6153872}
(n-1)\beta
&=\frac{1}{2n^2}
\sum_{(j,k)\in\set{n}_{\neq}^2}\sum_{(r,s)\in\set{n}_{\neq}^2}
(y_{j,r}-y_{k,r})z_{j,k,r,s}.
\end{align}
By interchanging $r$ with $s$, 
\begin{align}\label{e6153873}
(n-1)\beta&=-\frac{1}{2n^2}
\sum_{(j,k)\in\set{n}_{\neq}^2}\sum_{(r,s)\in\set{n}_{\neq}^2}
(y_{j,s}-y_{k,s})z_{j,k,r,s}
\end{align}
and adding the right-hand sides of \eqref{e6153872} and 
\eqref{e6153873}, 
\begin{align*}
(n-1)\beta
&=\frac{1}{4n^2}
\sum_{(j,k)\in\set{n}_{\neq}^2}\sum_{(r,s)\in\set{n}_{\neq}^2}
(y_{j,r}-y_{k,r}-y_{j,s}+y_{k,s})z_{j,k,r,s}
=\frac{1}{4n^2}
\sum_{(j,k)\in\set{n}_{\neq}^2}\sum_{(r,s)\in\set{n}_{\neq}^2}
z_{j,k,r,s}^2,
\end{align*}
which implies \eqref{e9726578}. 
\end{proof}

\begin{proof}[Proof of Lemma \ref{l1876565}]
Using Lemma \ref{l375876}\ref{l375876.c} with $g$ being defined there
together with \eqref{e872897}, 
we obtain
\begin{align*}
\gamma(x)
&=\frac{1}{n^2(n-1)}\sum_{(j,k)\in\set{n}_{\neq}^2}
\sum_{(r,s)\in\set{n}_{\neq}^2}
g(b_{j,k,r,s},xb_{j,k,r,s})\\
&\geq \frac{1}{n^2(n-1)}\sum_{(j,k)\in\set{n}_{\neq}^2}
\sum_{(r,s)\in\set{n}_{\neq}^2}
\Bigl(b_{j,k,r,s}^2-\frac{4}{27x^2}\Bigr)
=4\Bigl(\sigma^2-\frac{n-1}{27x^2}\Bigr).
\end{align*}
This shows the first inequality in \eqref{e78257}. 
The second one is clear.
Let us now prove \eqref{e2976787}. 
Identity \eqref{e687165} and the Cauchy-Schwarz inequality imply that
\begin{align*}
\widetilde{\gamma}(x)
&=\frac{1}{n^2(n-1)}\sum_{(j,k,r,s)\in\set{n}^4}
b_{j,k,r,s}\widetilde{a}_{j,r}\min\{1,\abs{x\widetilde{a}_{j,r}}\}\\
&\leq\Bigl(\frac{1}{n^2(n-1)}\sum_{(j,k)\in\set{n}_{\neq}^2}
\sum_{(r,s)\in\set{n}_{\neq}^2}
b_{j,k,r,s}^2\min\{1,\abs{x\widetilde{a}_{j,r}}\}\Bigr)^{1/2}
\sqrt{\widetilde{\gamma}(x)}.
\end{align*}
Therefore 
\begin{align*}
\widetilde{\gamma}(x)
&\leq \frac{1}{n^2(n-1)}\sum_{(j,k)\in\set{n}_{\neq}^2}
\sum_{(r,s)\in\set{n}_{\neq}^2}
b_{j,k,r,s}^2\min\{1,\abs{x\widetilde{a}_{j,r}}\}
\leq T_1+T_2,
\end{align*}
where 
\begin{align*}
T_1
&:=\frac{y^2}{n^2(n-1)}\sum_{(j,k)\in\set{n}_{\neq}^2}
\sum_{(r,s)\in\set{n}_{\neq}^2}
\widetilde{a}_{j,r}^2\min\{1,\abs{x\widetilde{a}_{j,r}}\}
\bbone_{[0,y\abs{\widetilde{a}_{j,r}}]}(\abs{b_{j,k,r,s}})\\
&\leq \frac{y^2(n-1)}{n^2}\sum_{(j,r)\in\set{n}^2}
\widetilde{a}_{j,r}^2\min\{1,\abs{x\widetilde{a}_{j,r}}\} 
=y^2\Bigl(\frac{n-1}{n}\Bigr)^2\widetilde{\gamma}(x)
\end{align*}
and
\begin{align*}
T_2
&:=\frac{1}{n^2(n-1)}\sum_{(j,k)\in\set{n}_{\neq}^2}
\sum_{(r,s)\in\set{n}_{\neq}^2}
b_{j,k,r,s}^2\min\Bigl\{1,\frac{1}{y}\abs{xb_{j,k,r,s}}\Bigr\}
\bbone_{(y\abs{\widetilde{a}_{j,r}},\infty)}(\abs{b_{j,k,r,s}})\\
&\leq \frac{1}{n^2(n-1)}\sum_{(j,k)\in\set{n}_{\neq}^2}
\sum_{(r,s)\in\set{n}_{\neq}^2}
b_{j,k,r,s}^2\min\Bigl\{1,\frac{1}{y}\abs{xb_{j,k,r,s}}\Bigr\}
=\gamma\Bigl(\frac{x}{y}\Bigr),
\end{align*}
giving the first inequality in \eqref{e2976787}. 
Using Lemma \ref{l375876}\ref{l375876.a}
and  (\ref{e97257}), we obtain
\begin{align}
\lefteqn{n^2(n-1)\gamma(x) 
=\sum_{(j,k)\in\set{n}_{\neq}^2}\sum_{(r,s)\in\set{n}_{\neq}^2}
g(b_{j,k,r,s},xb_{j,k,r,s}) }\nonumber\\
&\leq \sum_{(j,k)\in\set{n}_{\neq}^2}\sum_{(r,s)\in\set{n}_{\neq}^2}
(g(b_{j,k,r,s},x\widetilde{a}_{j,r})
+g(b_{j,k,r,s},x\widetilde{a}_{k,r})
+g(b_{j,k,r,s},x\widetilde{a}_{j,s})
+g(b_{j,k,r,s},x\widetilde{a}_{k,s})
)\nonumber\\
&\leq 4T_3, \label{e987265}
\end{align}
where 
\begin{align*}
\lefteqn{T_3
:=\sum_{(j,k)\in\set{n}_{\neq}^2}\sum_{(r,s)\in\set{n}_{\neq}^2}
g(b_{j,k,r,s},x\widetilde{a}_{j,r})
=\sum_{(j,k)\in\set{n}^2}\sum_{(r,s)\in\set{n}^2}
(\widetilde{a}_{j,r}-\widetilde{a}_{k,r}-\widetilde{a}_{j,s}
+\widetilde{a}_{k,s})^2\min\{1,x\abs{\widetilde{a}_{j,r}}\}}\\
&=\sum_{(j,k)\in\set{n}^2}\sum_{(r,s)\in\set{n}^2}
(g(\widetilde{a}_{j,r},x\widetilde{a}_{j,r})
+g(\widetilde{a}_{k,r},x\widetilde{a}_{j,r})
+g(\widetilde{a}_{j,s},x\widetilde{a}_{j,r})
+g(\widetilde{a}_{k,s},x\widetilde{a}_{j,r}))\\
&=\sum_{(j,k)\in\set{n}^2}\sum_{(r,s)\in\set{n}^2}
\Bigl(g(\widetilde{a}_{j,r},x\widetilde{a}_{j,r})
+\frac{1}{2}(g(\widetilde{a}_{k,r},x\widetilde{a}_{j,r})
+g(\widetilde{a}_{j,r},x\widetilde{a}_{k,r}))\\
&\quad{}+\frac{1}{2}(g(\widetilde{a}_{j,s},x\widetilde{a}_{j,r})
+g(\widetilde{a}_{j,r},x\widetilde{a}_{j,s}))
+\frac{1}{2}(g(\widetilde{a}_{k,s},x\widetilde{a}_{j,r}) 
+g(\widetilde{a}_{j,r},x\widetilde{a}_{k,s}))\Bigr) \hspace{3.5cm}
\end{align*}
and hence, by Lemma \ref{l375876}\ref{l375876.b},
\begin{align*}
T_3
&\leq\sum_{(j,k)\in\set{n}^2}\sum_{(r,s)\in\set{n}^2}
\Bigl(g(\widetilde{a}_{j,r},x\widetilde{a}_{j,r})
+\frac{1}{2}(g(\widetilde{a}_{k,r},x\widetilde{a}_{k,r})
+g(\widetilde{a}_{j,r},x\widetilde{a}_{j,r}))\\
&\quad{}
+\frac{1}{2}(g(\widetilde{a}_{j,s},x\widetilde{a}_{j,s})
+g(\widetilde{a}_{j,r},x\widetilde{a}_{j,r}))
+\frac{1}{2}(g(\widetilde{a}_{k,s},x\widetilde{a}_{k,s})
+g(\widetilde{a}_{j,r},x\widetilde{a}_{j,r}))
\Bigr)\\
&=4\sum_{(j,k)\in\set{n}^2}\sum_{(r,s)\in\set{n}^2}
g(\widetilde{a}_{j,r},x\widetilde{a}_{j,r})
=4n^2\sum_{(j,r)\in\set{n}^2}
\widetilde{a}_{j,r}^2\min\{1,\abs{x\widetilde{a}_{j,r}}\}
=4n^2(n-1)\widetilde{\gamma}(x).
\end{align*}
This, together with \eqref{e987265},
implies the second inequality in \eqref{e2976787}.
\end{proof}

\begin{proof}[Proof of Lemma \ref{l28766}]
Let $T$ denote the left-hand side of the inequality in 
\eqref{e218153}. Then it is easily seen that
\begin{align*}
\lefteqn{n^2(n-1) T}\\
&=\sum_{(j,k)\in\set{n}_{\neq}^2}\sum_{(r,s)\in\set{n}_{\neq}^2}
b_{j,k,r,s}^2\int_0^\infty\int_0^1tu
\exp\Bigl(-ct^2u^2-(1-u^2)\frac{t^2}{2}\Bigr)
\min\{1,\abs{xtub_{j,k,r,s}}\}\,\dd u\,\dd t\\
&\leq \sum_{(j,k)\in\set{n}_{\neq}^2}
\sum_{(r,s)\in\set{n}_{\neq}^2}b_{j,k,r,s}^2 
\min\{I_1(t),\abs{xb_{j,k,r,s}}I_2(t)\},
\end{align*}
where
\begin{align*}
I_1(t)
&=\int_0^1\int_0^\infty tu
\exp\Bigl(-ct^2u^2-(1-u^2)\frac{t^2}{2}\Bigr)\,\dd t\,\dd u
=\frac{-\log(2c)}{2\widetilde{c}^2},\\
I_2(t)
&=\int_0^1\int_0^\infty(tu)^2
\exp\Bigl(-ct^2u^2-(1-u^2)\frac{t^2}{2}\Bigr)\,\dd t\,\dd u
=\frac{\sqrt{\pii}}{2\widetilde{c}^2\sqrt{c}}
\Bigl(1-\frac{\sqrt{2c}}{\widetilde{c}}\arcsin(\widetilde{c})\Bigr),
\end{align*}
from which \eqref{e218153} follows. 
\end{proof}

\section*{Acknowledgment}
I thank Andrew Barbour for discussing with me in 2005 the problem to 
prove Bolthausen's theorem with an explicit constant. I also thank 
Lutz Mattner for some suggestions, which helped to improve a previous 
version of this paper. 


{\small 
\let\oldbibliography\thebibliography
\renewcommand{\thebibliography}[1]{\oldbibliography{#1}
\setlength{\itemsep}{0.8ex plus0.8ex minus0.8ex}} 
\linespread{1}
\selectfont
\bibliography{cclt_85}
}
\end{document}